\documentclass[12pt]{amsart}
\usepackage{amscd,verbatim}
\usepackage{mathrsfs,yfonts,times}
\usepackage{amssymb,array}

\usepackage[all]{xy}

\newcommand{\G}{\mathbb{G}}

\renewcommand{\P}{\mathbb{P}}
\newcommand{\Q}{\mathbb{Q}}
\newcommand{\Z}{\mathbb{Z}}
\newcommand{\R}{\mathbb{R}}
\newcommand{\C}{\mathbb{C}}

\newcommand{\sE}{\mathcal{E}}

\newcommand{\Coker}{\operatorname{Coker}}

\newcommand{\Pic}{\operatorname{Pic}}

\newcommand{\Tor}{\operatorname{Tor}}
\newcommand{\Div}{\operatorname{Div}}
\newcommand{\ord}{\operatorname{ord}}
\newcommand{\Res}{\operatorname{Res}}
\newcommand{\Gal}{\operatorname{Gal}}

\newcommand{\ol}{\overline}

\newcommand{\sC}{\mathscr{C}}
\newcommand{\sD}{\mathscr{D}}
\newcommand{\sP}{\mathscr{P}}
\newcommand{\CC}{\mathscr{C}}
\newcommand{\DD}{\mathscr{D}}
\newcommand{\PP}{\mathscr{P}}

\def\SL{\mathrm{SL}}

\def\M#1#2#3#4{\begin{pmatrix}#1&#2\\#3&#4\end{pmatrix}}
\def\SM#1#2#3#4{\left(\begin{smallmatrix}#1&#2\\#3&#4\end{smallmatrix}\right)}

\renewcommand{\epsilon}{\varepsilon}
\renewcommand{\div}{\operatorname{div}}

\newtheorem{lemma}{Lemma}[subsection]

\newtheorem{theorem}[lemma]{Theorem}

\newtheorem{proposition}[lemma]{Proposition}
\newtheorem{Proposition}[lemma]{Proposition}
\newtheorem{Lemma}[lemma]{Lemma}

\newtheorem{corollary}[lemma]{Corollary}
\newtheorem{Corollary}[lemma]{Corollary}

\theoremstyle{remark}

\newtheorem{remark}[lemma]{Remark}
\newtheorem{Remark}[lemma]{Remark}

\numberwithin{equation}{subsection}

\begin{document}
\title[Rational torsion on the generalized Jacobians]
{Rational torsion on the generalized Jacobian \\
of a modular curve with cuspidal modulus}

\author{Takao Yamazaki \and Yifan Yang}
\date{\today}
\address{Mathematical Institute, Tohoku University,
Aoba, Sendai 980-8578, Japan}
\email{ytakao@math.tohoku.ac.jp}
\address{Department of Applied Mathematics, 
National Chiao Tung University, Hsinchu 300, Taiwan}
\email{yfyang@math.nctu.edu.tw}

\begin{abstract}
We consider the generalized Jacobian $\widetilde{J}_0(N)$
of a modular curve $X_0(N)$ with respect to
a reduced divisor given by the sum of all cusps on it.
When $N$ is a power of a prime $\geq 5$,
we exhibit that the group of
rational torsion points $\widetilde{J}_0(N)(\Q)_{\Tor}$
tends to be much smaller than the classical Jacobian.
\end{abstract}

\thanks{The first author is supported by JSPS KAKENHI Grant
  (15K04773). The second author is supported by Grant
  102-2115-M-009-001-MY4 of the Ministry of Science and Technology,
  Taiwan (R.O.C.)}
\maketitle

\section{Introduction}

\subsection{}\label{sect:starter}
Let $N$ be a natural number 
and let $X_0(N)$ be the modular curve 
with respect to 
$\Gamma_0(N) = 
\{ (\begin{smallmatrix} a & b \\ c & d \end{smallmatrix})
\in SL_2(\Z)
\mid c \equiv 0 \bmod N \}$,
which we regard as a smooth projective curve over $\Q$.
Its Jacobian variety $J_0(N)$ is an important object in arithmetic geometry
and is intensively studied by many authors.
By the Mordell-Weil theorem, the group $J_0(N)(\Q)$ of $\Q$-rational
points on $J_0(N)$ is finitely generated, and hence its torsion subgroup 
$J_0(N)(\Q)_{\Tor}$ is finite. The torsion subgroup contains an
important subgroup $\CC(N)$ generated by classes of $\Q$-rational
divisors of degree $0$ with support on cusps of $X_0(N)$, called the
\emph{$\Q$-rational cuspidal divisor class group}. (By Manin's theorem
\cite{Manin}, divisors with support on cusps are of finite order in
$J_0(N)$.) Ogg \cite{Ogg}
conjectured and later Mazur \cite{Mazur} proved that when $N=p$ is a
prime, the two groups $\CC(p)$ and $J_0(p)(\Q)_{\Tor}$ coincide and
are cyclic of order $(p-1)/(p-1,12)$. For general cases, it is still an open
problem whether the two groups are equal, although the works of
Lorenzini \cite{Lorenzini} and Ling \cite{Ling} have given a partially
affirmative answer to the problem. We summarize the results mentioned
above in the theorem below.
Here for $m \in \Z_{>0}$, we say two abelian groups are isomorphic 
{\it up to $m$-torsion} if they become isomorphic
after tensoring with $\Z[1/m]$.

\begin{theorem}\label{thm:c-and-j}
Let $p$ be a prime number and set
$a:=(p-1)/(p-1, 12), ~b:=(p+1)/(p+1, 12)$.
Let $n$ be a positive integer.
\begin{enumerate}
\item
If $N=p$, then $J_0(p)(\Q)_{\Tor}=\CC(p)$ and
is a cyclic group of order $a$.
(Mazur \cite[Theorem 1]{Mazur}.)
\item
Suppose $p \not\equiv 11 \bmod 12$. If $p \geq 5$ and $N=p^n$, then
the three groups $J_0(p^n)(\Q)_{\Tor}$, $\CC(p^n)$, and $(\Z/a\Z)^n
\times (\Z/b\Z)^{n-1}$ are isomorphic up to $2p$-torsion.
(Lorenzini \cite[Theorem 4.6]{Lorenzini}.)
\item
The previous statement (2) holds
without the assumption $p \not\equiv 11 \bmod 12$
but up to $6p$-torsion
(Ling \cite[Theorem 4]{Ling}).
\item Assume that $p\ge 5$. If $n$ is even, then
$$
  \CC(p^n)\simeq(\Z/a\Z)^n\times(\Z/b\Z)^{n-1}\times
  \prod_{i=n/2}^{n-2}\Z/p^i\Z\times
  \prod_{i=(n/2)+1}^{n-1}\Z/p^i\Z.
$$
If $n$ is odd, then
$$
  \CC(p^n)\simeq(\Z/a\Z)^n\times(\Z/b\Z)^{n-1}\times
  \prod_{i=(n+1)/2}^{n-2}\Z/p^i\Z\times
  \prod_{i=(n+1)/2}^{n-1}\Z/p^i\Z.
$$
In particular, the order of $\CC(p^n)$ is $a^nb^{n-1}p^{k_n}$, where
\begin{equation*}
  k_n=\begin{cases}
    (n-2)(3n-2)/4, &\text{if }n\text{ is even}, \\
    (n-1)(3n-5)/4, &\text{if }n\text{ is odd}. \end{cases}
\end{equation*}
(Ling \cite[Theorem 1]{Ling}.)
\end{enumerate}
\end{theorem}

\begin{remark}\label{rem:conj}
Recently, Ohta \cite{Ohta} proved that
$\sC(N)$ and $J_0(N)(\Q)_{\Tor}$ are isomorphic up to $2$-torsion
when $N$ is the product of distinct odd primes.
\end{remark}

Let $C_0(N)$ be the closed subset of $X_0(N)$ consisting of all cusps.
We regard $C_0(N)$ as an effective reduced divisor on $X_0(N)$. In
this paper, we consider the \emph{generalized Jacobian} $\widetilde{J}_0(N)$
of $X_0(N)$ with modulus $C_0(N)$ in the sense of Rosenlicht-Serre
\cite{Serre}. It should be as important as $J_0(N)$ in arithmetic
geometry of modular curves, but somehow $\widetilde{J}_0(N)$ has not
been studied much. We are interested in the group of $\Q$-rational
points $\widetilde{J}_0(N)(\Q)$ on $\widetilde{J}_0(N)$.
Although it is not finitely generated (unless $N=1$),
its torsion subgroup $\widetilde{J}_0(N)(\Q)_{\Tor}$ is finite.
In this paper we observe that 
$\widetilde{J}_0(N)(\Q)_{\Tor}$ is unexpectedly
smaller than $J_0(N)(\Q)_{\Tor}$
by proving the following result,
which shows a sharp contrast with Theorem \ref{thm:c-and-j}. (For
example, Mazur's theorem shows that the cardinality of
$J_0(p)(\Q)_{\Tor}$ grows linearly as $p$ increases, but on the
contrary, our result shows that the cardinality of $\widetilde
J_0(p)_{\Tor}$ remains the same for all $p$.)

\begin{theorem}\label{thm:intro}
Let $p$ be a prime number
and $n$ be a positive integer.
\begin{enumerate}
\item
If $N=p$, then $\widetilde{J}_0(p)(\Q)_{\Tor}$
is a cyclic group of order $2$.
\item
Suppose $p \not\equiv 11 \bmod 12$.
If $p \geq 5$ and $N=p^n$, then 
$\widetilde{J}_0(p^n)(\Q)_{\Tor}$ is isomorphic to 
the trivial group up to $2p$-torsion.
\item
The previous statement (2) holds
without the assumption $p \not\equiv 11 \bmod 12$
but up to $6p$-torsion.
\item Assume that $p\ge 5$. Suppose that the conjecture
  $J_0(p^n)(\Q)_{\Tor}=\CC(p^n)$ is true. If $n$ is even, then 
$$
  \widetilde J_0(p^n)(\Q)_{\Tor}\simeq
  \prod_{i=0}^{(n/2)-1}\Z/(2p^i\Z)\times\prod_{i=1}^{n/2}\Z/(2p^i\Z).
$$
If $n$ is odd, then
$$
  \widetilde J_0(p^n)(\Q)_{\Tor}\simeq
  \prod_{i=0}^{(n-1)/2}\Z/(2p^i\Z)\times\prod_{i=1}^{(n-1)/2}\Z/(2p^i\Z).
$$
\end{enumerate}
\end{theorem}

This result is actually a consequence of our main theorem
(Theorem \ref{thm:main}) below. However, before we state our main
result, let us pause here to recall some basic facts
about generalized Jacobian (cf. \cite{Serre}).


\subsection{}\label{sect:as-small}
Let $C$ be a smooth projective geometrically connected curve over a field $k$,
and $J$ the Jacobian variety of $C$.
We give ourselves distinct closed points $P_0, \dots, P_n \in C$.
We assume that $P_n$ is a $k$-rational point.
We consider the generalized Jacobian 
 $\widetilde{J}$ of $C$ with modulus $D=P_0+\dots+P_n$.
There is an exact sequence
\[ 0 \to \G_m 
\to \bigoplus_{i=0}^n \Res_{k(P_i)/k} \G_m \to \widetilde{J} \to J \to 0 
\]
of commutative algebraic groups over $k$.
Here $\Res_{k(P_i)/k}$ denotes the Weil restriction.
We have $\Res_{k(P_n)/k} \G_m=\G_m$ since $P_n$ is a $k$-rational point.
As we have $H^1(k, \Res_{k(P_i)/k} \G_m)=0$ 
(by Hilbert 90 and Szpiro's lemma), 
it induces exact sequences of abelian groups 
\begin{equation}\label{eq:ex-seq0}
 0 \to \bigoplus_{i=0}^{n-1} k(P_i)^\times
\overset{\iota}{\to} \widetilde{J}(k) \to J(k) \to 0 
\end{equation}
and
\begin{equation}\label{eq:ex-seq1}
0 \to \bigoplus_{i=0}^{n-1} \mu(k(P_i)) 
\to 
\widetilde{J}(k)_{\Tor} 
\overset{\rho}{\to} 
J(k)_{\Tor}
\overset{\delta}{\to} 
\bigoplus_{i=0}^{n-1} k(P_i)^\times \otimes \Q/\Z,
\end{equation}
where we denote by
$\mu(F)$ the group of all roots of unity in $F$ for a field $F$.



\begin{remark}
Consider an effective divisor $D'$ which has the same support as $D$
(that is, $D' = \sum_{i=0}^n a_i P_i$ with $a_i \in \Z_{>0}$).
One can consider the generalized Jacobian $J'$ 
of $C$ with modulus $D'$.
Then there is a canonical surjection $J' \to \widetilde{J}$
whose kernel is unipotent.
In particular, 
when $k$ is of characteristic zero, 
we have an isomorphism 
$J'(k)_{\Tor} \to \widetilde{J}(k)_{\Tor}$
and hence there is nothing new in our problem.
\end{remark}

\subsection{}\label{sect:cuspidal}
We return to the setting in \S \ref{sect:starter}. 
Let $p$ be a prime number and let $n$ be a positive integer.
Then $C_0(p^n)$ consists of $n+1$ points $P_0, \dots, P_n$,
which we will arrange in such a way that the residue field $\Q(P_i)$
of $P_i$ is the cyclotomic field $\Q(\mu_{p^{d(i)}})$
of degree $p^{d(i)}$ with $d(i):=\min(i, n-i)$ for each $i=0, \dots,
n$. (See \S \ref{sect:cusps} for more details.)
In particular, $P_0$ and $P_n$ are $\Q$-rational. Then the map
$\delta$ in \eqref{eq:ex-seq1} for $(k, C, D)=(\Q, X_0(p^n), C_0(p^n))$
reads
\begin{equation}\label{eq:delta0}
\delta:J_0(p^n)(\Q)_{\Tor} \to \bigoplus_{i=0}^{n-1} 
\Q(\mu_{p^{d(i)}})^\times \otimes \Q/\Z,
\quad d(i)=\min(i, n-i).
\end{equation}
Thus, Theorem \ref{thm:intro} follow from 
Theorem \ref{thm:c-and-j} and the following theorem,
which is the main result of this article.




\begin{theorem}\label{thm:main}
Let $p \geq 5$ be a prime number and 
let $n$ be a positive integer.
Then the restriction of the map \eqref{eq:delta0}
to $\sC(p^n)$ is injective.
\end{theorem}

The proof of Theorem \ref{thm:main} will occupy 
almost all of the rest of this article
and will be completed in \S \ref{sect:proof-pn}.
On the other hand, Theorem \ref{thm:main} does not admit 
a naive generalization to other values of level $N$.
Indeed, in the last section \S \ref{sect:level-pq}
we shall observe the following result:

\begin{proposition}\label{prop:intro-pq}
Let $p, q$ be two distinct prime numbers.
(Then $C_0(pq)$ consists of four $\Q$-rational points.)
If $p \equiv q \equiv 1 \bmod 12$, 
then the kernel of the restriction 
to $\sC(pq)$ of 
\[ \delta :J_0(pq)(\Q)_{\Tor} \to
(\Q^\times)^3 \otimes \Q/\Z
\]
is a cyclic group of order $(p-1)(q-1)/24$.
\end{proposition}
In view of Ohta's result (see Remark \ref{rem:conj}),
we find that $\widetilde{J}_0(pq)(\Q)_{\Tor}$
is isomorphic to a cyclic group of order $(p-1)(q-1)/3$
up to $2$-torsion.
This shows another sharp contrast with Theorem \ref{thm:intro}.
We are not able (but hoping) to find
more conceptual reason for such difference.

\subsection*{Notation}
Let $\ol{\Q}$ be an algebraic closure of $\Q$
and fix an embedding $\ol{\Q} \hookrightarrow \C$.
For $m \in \Z_{>0}$, 
we set $\zeta_m := e^{2 \pi i/m} \in \ol{\Q}^\times$
and $\mu_m:=\{ \zeta_m^k \mid k \in \Z \} \subset \ol{\Q}^\times$.
For an abelian group $A$,
we write $A_{\Tor}$ for the subgroup of torsion elements of $A$.
For a field $F$, we write $\mu(F)=(F^\times)_{\Tor}$.

\section{Torsion rational points on generalized Jacobian}
\subsection{}\label{sect:delta-intro}
In this section, we use the notations introduced in
\S \ref{sect:as-small}.
We always assume $P_n$ is $k$-rational.
We will give an explicit description of the map $\delta$
in \eqref{eq:ex-seq1}
in Lemma \ref{lem:evaluation-delta} below.
Take $x \in \bigoplus_{i=0}^{n-1} k(P_i)^\times$, $m \in \Z_{>0}$
and $a \in J(k)$ such that $ma=0$.
Then by definition we have
$\delta(a)=x \otimes \frac{1}{m}$
if there is a lift $\widetilde{a} \in \widetilde{J}(k)$ of $a$
such that $\iota(x)=m \widetilde{a}$,
where $\iota$ is the map appearing in \eqref{eq:ex-seq0}.

\subsection{}
We recall some basic facts about the relative Picard group
and generalized Jacobian (cf. \cite[Chapter V]{Serre}).
Denote by $K$ the function field of $C$.
For a closed point $P$ on $C$,
we write $K_P$ for the completion of $K$ at $P$,
$O_P$ for the ring of integers in $K_P$,
$t_P \in O_P$ for a (fixed) uniformizer,
$U_P :=(1+t_P O_P)^\times$ for the group of principal units in $O_P$,
and $k(P):=O_P/t_PO_P$ for the residue field at $P$.

Let $U := C \setminus |D|$
be the open complement of the divisor $D=P_0 + \dots + P_n$.
Let us consider the abelian group
\[ \Div(C, D) := \Div(U)
\oplus \bigoplus_{i=0}^n (K_{P_i}^\times/U_{P_i}).
\]
We have a canonical map
\[ K^\times \to \Div(C, D), \quad
f \mapsto \Big(\div_U(f); (f \bmod U_{P_i})_{i=0}^n)\Big),
\]
whose cokernel is by definition
the {\it relative Picard group}
$\Pic(C, D)$ of $C$ relative to $D$.
We also have a commutative diagram with exact rows
\[
\xymatrix{
0
\ar[r]
&
K^\times
\ar[r]
\ar@{=}[d]
&
\Div(C, D)
\ar[r]
\ar@{->>}[d]_{\alpha}
&
\Pic(C, D)
\ar[r]
\ar@{->>}[d]_{\ol{\alpha}}
&
0
\\
k^\times
\ar@{^{(}->}[r]
&
K^\times
\ar[r]
&
\Div(C)
\ar[r]
&
\Pic(C)
\ar[r]
&
0,
}
\]
where $\alpha$ a canonical surjection given by
\[ \Big(E; (f_i \bmod U_{P_i})_{i=0}^n \Big) \mapsto E + \sum_{i=0}^n \ord_{P_i}(f) P_i,
\]
and $\ol{\alpha}$ is induced by $\alpha$.
Combined with isomorphisms
\begin{gather*}
\ker(\alpha) \cong \bigoplus_{i=0}^n (O_{P_i}^\times/U_{P_i})
\cong \bigoplus_{i=0}^n k(P_i)^\times,
\\
\bigoplus_{i=0}^n k(P_i)^\times/(k^\times) \cong 
 \bigoplus_{i=0}^{n-1} k(P_i)^\times,
\quad
(c_i)_{i=0}^n \bmod (k^\times) \mapsto (c_i/c_n)_{i=0}^{n-1},
\end{gather*}
where $(k^\times)$ is the image of 
the diagonal map $k^\times \to \oplus_i~ k(P_i)^\times$,
we obtain an exact sequence
\begin{equation}\label{eq:pic-ex}
 0 \to \bigoplus_{i=0}^{n-1} k(P_i)^\times \to \Pic(C, D) \to \Pic(C) \to 0.
\end{equation}

On the other hand,
there are  canonical isomorphisms
\begin{align*}
&{J}(k) \cong \ker[\Pic(C) \overset{\deg}{\to} \Z],
\\
&\widetilde{J}(k) \cong \ker[\Pic(C, D) \overset{\ol{\alpha}}{\to}
\Pic(C) \overset{\deg}{\to} \Z].
\end{align*}
Then \eqref{eq:ex-seq0} is deduced from \eqref{eq:pic-ex} by restriction.
We also obtain
\[J(k)_{\Tor} \cong \Pic(C)_{\Tor},
\quad
\widetilde{J}(k)_{\Tor} \cong \Pic(C, D)_{\Tor}.
\]

\subsection{}
We are now ready to describe explicitly 
the map $\delta$ from \eqref{eq:ex-seq1}.

\begin{lemma}\label{lem:evaluation-delta}
Let $E=\sum_{i=0}^n a_i P_i \in \Div^0(C)$ 
be a degree zero divisor supported on $D$.
Suppose that its class $[E]$ in $J(k)$ is killed by $m \in \Z_{>0}$
so that there is $f \in K^\times$ such that $\div_C(f)=mE$.
Define
\[ 
\sE:=
\Big( (\frac{f}{t_{P_n}^{m a_n}})(P_n) 
(\frac{t_{P_i}^{m a_i}}{f})(P_i) \Big)_{i=0}^{n-1}
\in \bigoplus_{i=0}^{n-1} k(P_i)^\times.
\] 
Then we have
\[ \delta([E]) = 
\sE \otimes \frac{1}{m}
\quad \text{ in } \quad
\bigoplus_{i=0}^{n-1} k(P_i)^\times \otimes \Q/\Z.
\] 
(Note that $\sE \otimes \frac{1}{m}$
does not depend on the choices of $t_{P_i}$ and $f$.)
\end{lemma}

\begin{proof}
We use the fact recalled in \S \ref{sect:delta-intro}.
Put $\widetilde{E}:= (0; (t_{P_i}^{a_i})_{i=0}^n) \in \Div(C, D)$
so that $\alpha(\widetilde{E})=E$.
It suffices to prove that 
the class of $m\widetilde{E}$ in $\Pic(C, D)$
is the same as $\iota(\sE) \in \widetilde{J}(k) \subset \Pic(C, D)$
(see \eqref{eq:ex-seq0} for the map $\iota$).
By definition,
$\iota(\mathcal{E})$ is given by the class of 
$(0; (f^{-1} t_{P_i}^{ma_i})_{i=0}^n)$.
Since $\div_C(f)=mE$, we have $\div_U(f)=0$,
and hence the class of
$(0; (f^{-1} t_{P_i}^{ma_i})_{i=0}^n)$
agrees with
that of 
$(0; (t_{P_i}^{ma_i})_{i=0}^n)=m \widetilde{E}$ in $\Pic(C, D)$.
We are done.
\end{proof}


\section{Preliminaries on modular curves}
\subsection{}\label{sect:cusps}
We return to the setting in \S \ref{sect:starter}.
We take an integer $N>1$ 
and consider the modular curve $X_0(N)$.
Recall that $C_0(N)$ denotes the set of cusps on $X_0(N)$
so that we have a canonical bijection
$C_0(N)(\C) \cong \Gamma_0(N)\backslash \P^1(\Q)$.
For each divisor $d>0$ of $N$,
there is a unique $Q_d \in C_0(N)$
such that the set of $\C$-rational points lying over $Q_d$
is given by $\Gamma_0(N)$-orbits of
$a/d \in \P^1(\Q)$ with $a \in \Z, ~(a,d)=1$.
We call $Q_d \in C_0(N)$ the cusp of \emph{level} $d$.
The residue field of $Q_d$ is $\Q(\zeta_m)$ with $m=(d, N/d)$,
hence the degree of $Q_d$ is $\phi(m)$,
where $\phi$ denotes the Euler function.
The classes of $0$ and $\infty \in \P^1(\Q)$ 
are $\Q$-rational and are of level $1$ and $N$ respectively.

We define
\begin{align}
\notag 
&\begin{aligned}
\DD(N) &:= \{ E \in \Div^0(X_0(N)) \mid |E| \subset C_0(N) \}
\\
&=\langle Q_d-\phi((d,N/d))Q_N \mid d|N \rangle,
\end{aligned}
\\
\label{eq:def-pp}
&\PP(N) := \DD(N) \cap \div(\Q(X_0(N))^\times),
\\
\notag 
&\CC(N) := \DD(N)/\PP(N).
\end{align}
As we recalled in the introduction,
$\sC(N)$ is a subgroup of $J_0(N)(\Q)_{\Tor}$,
hence finite.


\subsection{} We will use the Dedekind eta function to construct
modular functions needed for our purpose. Here let us recall some
well-known properties of the Dedekind eta function $\eta(\tau)$,
where as usual $\tau$ is a variable on the upper half plane $\mathbb{H}$.
We shall make use of the standard identification 
$X_0(N)(\C) \cong \Gamma_0(N)\backslash(\mathbb{H} \cup \P^1(\Q))$.

\begin{Proposition}[{\cite[Proposition 3.2.1]{Ligozat}}] \label{proposition: conditions for eta}
  Let $N$ be a positive integer. 
  The product
  $h=\prod_{\delta|N}\eta(\delta\tau)^{r_\delta}$, $r_\delta\in\Z$, is a
  modular function on $X_0(N)$ 
  if and only if the following conditions
  are satisfied:
  \begin{enumerate}
  \item $\sum_{\delta|N}r_\delta=0$,
  \item $\prod_{\delta|N}\delta^{r_\delta}$ is the square of a rational number,
  \item $\sum_{\delta|N}r_\delta\delta\equiv 0\mod 24$, and
  \item $\sum_{\delta|N}r_\delta(N/\delta)\equiv 0\mod 24$.
  \end{enumerate}
\end{Proposition}
We also remark that,
if these conditions are satisfied, then $h$ is defined over $\Q$
(see \cite[p. 32, Remarque]{Ligozat}).

\begin{Lemma}[{\cite[Proposition 3.2.8]{Ligozat}}]\label{lem:eta-prod-div}
  Let $N$ be a positive integer. Let $d$ and $\delta$ be positive
  divisor of $N$. Then the order of $\eta(\delta\tau)$ at a cusp of
  level $d$ 
  is $a_N(d,\delta)/24$, where
  $$
    a_N(d,\delta):=\frac N{(d,N/d)}\frac{(d,\delta)^2}{d\delta}.
  $$
  In particular, if
  $g(\tau)=\prod_{\delta|N}\eta(\delta\tau)^{r_\delta}$ is an
  eta-product satisfying the conditions in Proposition 
  \ref{proposition: conditions for eta}, then
  $$
    \div g=\frac1{24}\sum_{d, \delta|N}r_\delta a_N(d,\delta)(Q_d).
  $$
\end{Lemma}

When $N=p^n$ is a prime power, the orders of $\eta(p^k\tau)$ at cusps
can be summarized as follows.

\begin{Corollary} \label{corollary: order of eta}
Let $p^n$ be a prime power.
\begin{enumerate}
\item If $m\ge n/2$, then the order of $\eta(p^k\tau)$ at a cusp of
  level $p^m$ is
  $$
    \begin{cases}
    p^k/24, &\text{if }k\le m, \\
    p^{2m-k}/24, &\text{if }k>m. \end{cases}
  $$
\item If $m<n/2$, then the order of $\eta(p^k\tau)$ at a cusp of level
  $p^m$ is
  $$
    \begin{cases}
    p^{n-k}/24, &\text{if }m\le k, \\
    p^{n+k-2m}/24, &\text{if }m>k. \end{cases}
  $$
\end{enumerate}
\end{Corollary}

In order to obtain the Fourier expansion of an eta-product at a cusp,
we should need the following transformation formula for the Dedekind eta
function.

\begin{Lemma}[{\cite[pp.~125--127]{Weber}}] \label{lemma: eta} For
$$
  \gamma=\begin{pmatrix}a&b\\ c&d\end{pmatrix}\in SL_2(\mathbb Z),
$$
the transformation formula for $\eta(\tau)$ is given by, for $c=0$,
$$
  \eta(\tau+b)=e^{\pi ib/12}\eta(\tau),
$$
and, for $c\neq 0$,
$$
  \eta(\gamma\tau)=\epsilon(a,b,c,d)\sqrt{\frac{c\tau+d}i}\eta(\tau)
$$
with
\begin{equation*}
  \epsilon(a,b,c,d)=
  \begin{cases}\displaystyle
   \left(\frac dc\right)i^{(1-c)/2}
   e^{\pi i\left(bd(1-c^2)+c(a+d)\right)/12},
    &\text{if }c\text{ is odd},\\
  \displaystyle 
  \left(\frac cd\right)e^{\pi i\left(ac(1-d^2)+d(b-c+3)\right)/12},
    &\text{if }d\text{ is odd},
  \end{cases}
\end{equation*}
where $\displaystyle\left(\frac dc\right)$ is the Jacobi symbol.
\end{Lemma}

\section{Cuspidal divisor class group}\label{sect:cuspidal-pn}
\subsection{}
In \S \ref{sect:cuspidal-pn}--\ref{sect:proof-pn},
we consider the case $N=p^n$,
where $p$ is a prime greater than or equal to $5$ 
and $n$ is a positive integer.
We describe the group of modular units on $X_0(p^n)$ that gives
us $\PP(p^n)$ (see \eqref{eq:def-pp} for its definition).

\begin{Proposition} \label{proposition: generators of P}
  Let $p^n$ be a prime power with $p\ge 5$ prime and $n \in \Z_{>0}$. Then the
  group $\PP(p^n)$ is generated by the divisors of
  $$
    f(\tau)=\left(\frac{\eta(p\tau)}{\eta(\tau)}\right)^{24/(p-1,12)},
    \qquad
    g_k(\tau)=\frac{\eta(p^{k+2}\tau)}{\eta(p^k\tau)}, \quad
    k=0,\ldots,n-2.
  $$
\end{Proposition}

The proof of this proposition will be given in \S \ref{sect:pf-gen}.
We first deduce a corollary that will be used later.
For $i=0, \dots, n$,
we write $P_i:=Q_{p^i}$ for the cusp of level $p^i$
(see \S \ref{sect:cusps}).

\begin{corollary}\label{cor:gen}
Let $p, n$ be as in Proposition \ref{proposition: generators of P}.
\begin{enumerate}
\item 
The divisors $\div(f), \div(g_0), \dots, \div(g_{n-2})$
form a free $\Z$-basis of $\sP(p^n)$.
\item
Let $c_0, \dots, c_{n-2} \in \Z$
and write 
\[ \div(fg_0^{c_0} g_1^{c_1} \dots g_{n-2}^{c_{n-2}})=\sum_{i=0}^n s_i P_i,
 \quad  s_i \in \Z
\] 
in $\DD(p^n)$.
Then we have $(s_0, s_1, \dots, s_n)=(p-1)/(p-1, 12)$.
\end{enumerate}
\end{corollary}
\begin{proof}
(1) This is an immediate consequence of
Proposition \ref{proposition: generators of P},
since $\sP(p^n)$ is a free $\Z$-module of rank $n$
(as it is a finite index subgroup of $\sD(p^n)$.)

(2) 
Put $a:=(p-1)/(p-1, 12)$.
By using Corollary \ref{corollary: order of eta}
we first see that
the order of $g_k$ at any cusp is divisible by $a$
for $k=0, \dots, n-2$.
Hence, by (1), it suffices to show
the statement for $c_0=\dots=c_{n-2}=0$,
which again 
follows from Corollary \ref{corollary: order of eta}.
\end{proof}

In the proof of Proposition \ref{proposition: generators of P},
we use the following elementary lemma.
We omit its proof.

\begin{Lemma} \label{lemma: index of sublattice}
  Let $L_0\subset\R^{n+1}$ be the lattice of rank
  $n$ generated by the vectors of the form
  $(0,\ldots,1,-1,0,\ldots,0)$. Let $L_1$ be a sublattice of
  $L_0$ of the same rank generated by
  $v_1,\ldots,v_n\in L_1$. Let $v_{n+1}=(c_1,\ldots,c_{n+1})$ be
  any vector such that $\sum_i c_i\neq 0$, and $M$ be the
  $(n+1)\times(n+1)$ matrix whose $i$th row is $v_i$. Then we have
  $$
   (L_0:L_1)=\left|\left(\sum_{i=1}^{n+1}c_i\right)^{-1}
   \det M\right|.
  $$
\end{Lemma}

\subsection{Proof of Proposition \ref{proposition: generators of P}}\label{sect:pf-gen}
  Let $L_0$ be the lattice of rank $n$ 
  in $\Z^{n+1}=\bigoplus_{i=0}^n \Z e_i$ generated
  by vectors of the form $(0,\ldots,1,-1,0,\ldots,0)$.
  Recall that $D_i:=P_i - \phi((p^i, p^{n-i})) P_n ~(i=0, \dots, n-1)$
  form a $\Z$-basis of $\DD(p^n)$.
  Consider the natural group homomorphism $\lambda:\DD(p^n)\to\Z^{n+1}$
  defined by
  $$
    \lambda(D_i)=
    -\phi((p^i,p^{n-i}))e_0 + \phi((p^i,p^{n-i})) e_{i+1}
  \qquad (i=0, \dots, n-1). 
  $$
  Let $L_1=\lambda(\DD(p^n))$ be the image
  of $\DD(p^n)$ under $\lambda$. It is a sublattice of $L_0$.
  Let $\DD'$ be the group generated by the divisors of $f$
  and $g_k$ and $L_2$ be the image of $\DD'$ under $\lambda$.
  Then to show that the divisors of $f$ and $g_k$ generates
  $\PP(p^n)$, it suffices to show that the index of $L_2$ in
  $L_1$ is equal to the divisor class number given in Part (4)
  of Theorem \ref{thm:c-and-j}. To show that this indeed holds,
  we form $3$ square matrices $M$, $U$, and $V$ of dimension $n+1$. The
  first matrix $M=(M_{ij})_{i,j=0}^n$ is defined by
  \begin{equation*}
  \begin{split}
    M_{ij}&=\text{the order of }\eta(p^i\tau)\text{ at cusps of level
    }p^j \\
  &=\begin{cases}
    p^i/24, &\text{if }j\ge n/2\text{ and }i\le j, \\
    p^{2j-i}/24, &\text{if }j\ge n/2\text{ and }i>j, \\
    p^{n-i}/24, &\text{if }j<n/2\text{ and }i\ge j, \\
    p^{n+i-2j}/24, &\text{if }j<n/2\text{ and }i>j.
    \end{cases}
  \end{split}
  \end{equation*}
  The second matrix $U$ is a diagonal matrix with the diagonal
  entries being $\phi((p^i,p^{n-i}))$, $i=0,\ldots,n$. The third
  matrix $V$ is
  $$
    V=\begin{pmatrix}
    -c & c & 0 & 0 & \cdots & \cdots & 0 \\
    -1 & 0 & 1 & 0 & \cdots & \cdots & 0 \\
     0 &-1 & 0 & 1      & \cdots & \cdots & 0 \\
    \vdots & \vdots & & & & \vdots & \vdots \\
    0 & \cdots & \cdots & 0 & -1 & 0 & 1 \\
    1 & \cdots & \cdots & \cdots & 1 & 1 & 1 \end{pmatrix},
    \qquad c=\frac{24}{(p-1,12)}.
  $$
  That is, if we associate to an eta-product
  $\prod_{i=0}^n\eta(p^i\tau)^{r_i}$ a vector $(r_0,\ldots,r_n)$, then
  the first $n$ rows of $V$ are the vectors corresponding to the
  functions $f$ and $g_k$, while the last row of $V$ consists of
  $1$'s. Then the first $n$ rows of the matrix $VMU$ are precisely
  $\lambda(\div f)$ and $\lambda(\div g_k)$, $k=0,\ldots,n-2$.
  Now we claim that (see Theorem \ref{thm:c-and-j} for the definition of $a$ and $b$)
  \begin{enumerate}
  \item $\det V=24(n+1)/(p-1,12)$,
  \item $\displaystyle\det M
   =\begin{cases}
    (ab)^np^{(n-1)(3n-1)/4}/24, &\text{if }n\text{ is odd}, \\
    (ab)^np^{n(3n-4)/4}/24, &\text{if }n\text{ is even}.
    \end{cases}
    $
  \item $\det U=\prod_{i=0}^n\phi((p^i,p^{n-i}))$, and
  \item the sum of the entries in the last row of $VMU$ is $(n+1)p^{n-1}(p+1)/24$.
  \end{enumerate}
  Assuming that the claims are true for the moment, let us complete
  the proof of the proposition.

  It is clear that
  \begin{equation} \label{equation: Lambda1 index}
    (L_0:L_1)=\prod_{i=0}^n\phi((p^i,p^{n-i}))=\det U.
  \end{equation}
  By Lemma \ref{lemma: index of sublattice}, we have
  $$
    (L_0:L_2)=C^{-1}|\det(VMU)|,
  $$
  where $C$ is the sum of the entries in the last row of $VMU$. By the
  four claims above,
  $$
    \det(VMU)=\frac{(n+1)}{(p-1,12)}(ab)^n(\det U)\times
    \begin{cases}
    p^{(n-1)(3n-1)/4}, &\text{if }n\text{ is odd}, \\
    p^{n(3n-4)/4}, &\text{if }n\text{ is even}, \end{cases}
  $$
  and $C=(n+1)p^{n-1}(p+1)/24$. It follows that
  \begin{equation*}
  \begin{split}
    (L_0:L_2)=
    \frac{24(ab)^n\det U}{(p+1)(p-1,12)}\times
    \begin{cases}
    p^{(n-1)(3n-5)/4}, &\text{if }n\text{ is odd}, \\
    p^{(n-2)(3n-2)/4}, &\text{if }n\text{ is even}. \end{cases}
  \end{split}
  \end{equation*}
  Recall that the number $b$ is defined to be $(p+1)/(p+1,12)$. Also
  we may check case by case that $(p-1,12)(p+1,12)=24$. Therefore, the
  expression above can also be written as
  $$
    (L_0:L_2)
   =a^nb^{n-1}(\det U)\times
    \begin{cases}
    p^{(n-1)(3n-5)/4}, &\text{if }n\text{ is odd}, \\
    p^{(n-2)(3n-2)/4}, &\text{if }n\text{ is even}. \end{cases}
  $$
  Combining this with \eqref{equation: Lambda1 index}, we find that
  $$
    (L_1:L_2)=a^nb^{n-1}\times
    \begin{cases}
    p^{(n-1)(3n-5)/4}, &\text{if }n\text{ is odd}, \\
    p^{(n-2)(3n-2)/4}, &\text{if }n\text{ is even}, \end{cases}
  $$
  which agrees with the class number given in Part (4) of Theorem
  \ref{thm:c-and-j}. Therefore, we conclude that the divisors of
  $f$ and $g_k$, $k=0,\ldots,n-2$ generate $\PP(p^n)$. It remains to
  prove that the four claims are true.

  Claims (1) and (3) are obvious. To prove Claim (2), we start by
  giving examples. Consider the case $n=5$. The matrix $M$ in this
  case is
  $$
    \frac1{24}\begin{pmatrix}
    p^5 & p^3 & p   & 1   & 1   & 1 \\
    p^4 & p^4 & p^2 & p   & p   & p \\
    p^3 & p^3 & p^3 & p^2 & p^2 & p^2 \\
    p^2 & p^2 & p^2 & p^3 & p^3 & p^3 \\
    p   & p   & p   & p^2 & p^4 & p^4 \\
    1   & 1   & 1   & p   & p^3 & p^5 \end{pmatrix}
  $$
  We subtract the second column from the first column, the third
  column from the second column, the fourth column from the fifth
  column, and then the fifth column from the last column. The matrix
  becomes
  $$
    \frac1{24}\begin{pmatrix}
    p^3(p^2-1) & p(p^2-1)   & p   & 1 & 0 & 0 \\
    0          & p^2(p^2-1) & p^2 & p & 0 & 0 \\
    0          & 0          & p^3 & p^2 & 0 & 0 \\
    0          & 0          & p^2 & p^3 & 0 & 0 \\
    0          & 0          & p   & p^2 & p^2(p^2-1) & 0\\
    0          & 0          & 1   & p   & p(p^2-1) & p^3(p^2-1)
    \end{pmatrix},
  $$
  with the determinant unchanged. Thus,
  $$
    \det M=\frac1{24^6}p^{14}(p^2-1)^5=\frac1{24}(ab)^5p^{14}.
  $$
  In general, if $n$ is an odd integer greater than $3$, then a
  similar matrix manipulation (subtracting the second column from the
  first column, the third column from and etc.) will produce a matrix
  of the form
  $$
    \frac1{24}\begin{pmatrix}
    A_1 & B_1 & 0 \\
    0 & A_2 & 0 \\
    0 & B_2 & A_3 \end{pmatrix},
  $$
  where $A_1$ is an upper-triangular matrix of dimension $(n-1)/2$
  whose diagonal entries are
  $p^{n-2}(p^2-1),\ldots,p^{(n-1)/2}(p^2-1)$, $A_3$ is a
  lower-triangular matrix of the same dimension whose diagonal entries
  are $p^{(n-1)/2}(p^2-1),\ldots,p^{n-2}(p^2-1)$,
  $$
    A_2=\begin{pmatrix}
    p^{(n+1)/2} & p^{(n-1)/2} \\ p^{(n-1)/2} & p^{(n+1)/2} \end{pmatrix},
  $$
  and $B_i$ are some immaterial $(n-1)/2$-by-$2$ matrices. It follows
  that
  \begin{equation*}
  \begin{split}
    \det
    M&=\frac1{24^{n+1}}p^{2((n-1)/2+(n+1)/2+\cdots+(n-2))}(p^2-1)^{n-1}(p^{n+1}-p^{n-1})\\
   &=\frac1{24}(ab)^np^{(n-1)(3n-2)/4}.
  \end{split}
  \end{equation*}
  This proves Claim (2) for the case of odd $n$. The proof of the case
  of even $n$ is similar. For the case $n=4$, we have
  $$
    M=\frac1{24}\begin{pmatrix}
    p^4 & p^2 & 1  & 1   & 1 \\
    p^3 & p^3 & p  & p   & p \\
    p^2 & p^2 & p^2& p^2 & p^2 \\
    p   & p   & p  & p^3 & p \\
    1   & 1   & 1  & p^2 & 1 \end{pmatrix}.
  $$
  Subtracting the second column from the first column, the third
  column from the second column, the fourth column from the last
  column, and then the third column from the fourth column, we obtain
  the matrix
  $$
    \frac1{24}\begin{pmatrix}
    p^2(p^2-1) & p^2-1    & 1   & 0        & 0 \\
    0          & p(p^2-1) & p   & 0        & 0 \\
    0          & 0        & p^2 & 0        & 0 \\
    0          & 0        & p   & p(p^2-1) & 0 \\
    0          & 0        & 1   & p^2-1    & p^2(p-1) \end{pmatrix},
  $$
  whose determinant is
  $$
    \frac1{24^5}p^8(p^2-1)^4=\frac1{24}(ab)^4p^8.
  $$
  In general, a similar matrix manipulation yields a matrix of the form
  $$
    \frac1{24}\begin{pmatrix}
    A_1 & B_1 & 0 \\
    0 & p^{n/2} & 0 \\
    0 & B_2 & A_2 \end{pmatrix},
  $$
  where $A_1$ is an upper-triangular matrix of dimension $n/2$ with
  the diagonal entries being $p^{n-2}(p^2-1),\ldots,p^{n/2-1}$ and $A_2$ is a
  lower-triangular matrix of dimension $n/2$ with the diagonals being
  $p^{n/2-1},\ldots,p^{n-2}(p^2-1)$. Therefore,
  \begin{equation*}
  \begin{split}
    \det
    M&=\frac1{24^{n+1}}p^{2((n/2-1)+n/2+\cdots+(n-2))+n/2}(p^2-1)^n \\
  &=\frac1{24}(ab)^np^{n(3n-4)/4}.
  \end{split}
  \end{equation*}
  This completes the proof of Claim (2).

  To prove Claim (4), we first observe that since the last row of $V$
  consists of $1$'s, the sum of the entries in the last row of $VMU$
  is simply the sum of all entries in $VM$. Now the $(i,j)$-entry of
  $VM$ is the order of $\eta(p^{i-1}\tau)$ at a cusp of level
  $p^{j-1}$ times the number of such cusps. Therefore, the sum of the
  entries in the $i$th row of $VM$ is the degree of
  $\div\eta(p^{i-1}\tau)$, which is equal to
  $$
    \frac1{24}(\SL(2,\Z):\Gamma_0(p^n))=\frac1{24}p^{n-1}(p+1).
  $$
  (In general, the degree of a modular form of weight $k$ on
  $\Gamma_0(N)$ is $k(\SL(2,\Z):\Gamma_0(N))/12$. Here the weight of
  the Dedekind eta function is $1/2$.)
  Hence the sum of the entries in the last row of $VMU$ is
  $(n+1)p^{n-1}(p+1)/24$. This completes the proof of the proposition.
\qed

\section{Leading Fourier coefficients of modular units at cusps}
\subsection{}
In this section, we work out the leading Fourier coefficients of the
modular functions $f(\tau)$ and $g_k$ defined in Proposition
\ref{proposition: generators of P} at cusps.
To speak of such coefficients,
we first need to choose uniformizers at cusps.
Then the coefficients are canonically defined 
as an element of the residue fields of cusps,
but we can calculate it after base change to $\C$.
As cusps of the same
level are Galois conjugates, for our purpose, we only need to
calculate the leading coefficient at one of the $\C$-valued points of cusps of each level.
In general, to define a local uniformizer
at a cusp $\alpha \in \P^1(\Q)$ of $X_0(N)$, we choose an element $\sigma$ in
$\mathrm{GL}^+(2,\Q)$ such that $\sigma\infty=\alpha$. Let $h$
be the smallest positive integer such that
$$
  \sigma\M 1h01\sigma^{-1}\in\Gamma_0(N).
$$
Then a local uniformizer at $\alpha$ is
$$
  q_\alpha=e^{2\pi i\sigma^{-1}\tau/h}.
$$

\subsection{}
We return to the case $N=p^n$, 
where $p^n$ is a prime power with $p\ge 5$.
For convenience, 
our choice of a cusp $\alpha_m \in \P(\Q)$ of level $p^m$ is
\begin{equation} \label{equation: choice of alpha}
  \alpha_m=\begin{cases}
  1/p^m, &\text{if }m\ge n/2, \\
  -1/p^m, &\text{if }m<n/2. \end{cases}
\end{equation}
Then we can choose
$\sigma_{m}$ to be 
\begin{equation} \label{equation: choice of sigma}
  \sigma_{m}=\begin{cases}
  \SM 10{p^m}1, &\text{if }m\ge n/2, \\
  \SM0{-1}{p^n}0\SM10{p^{n-m}}1=\SM{-p^{n-m}}{-1}{p^n}0, &\text{if }m<n/2. \end{cases}
\end{equation}
We summarize our discussion as a lemma:
\begin{Lemma} \label{lemma: uniformizers}
  With the choice of $\sigma_{m}$ given in
  \eqref{equation: choice of sigma}, the local uniformizer at the cusps
  $\alpha_m$ in \eqref{equation: choice of alpha} is
  $$
    e^{2\pi i\sigma_{m}^{-1}\tau}
  $$
  for each $m$.
\end{Lemma}

\subsection{} 
In the following lemma, we adopt the following notation
$$
  f(\tau)\big|\M abcd:=f\left(\frac{a\tau+b}{c\tau+d}\right),
$$
which is slightly different from the usual meaning of the slash
operator.

\begin{Lemma} \label{lemma: eta at cusps}
  Let $p$ be an odd prime.
  \begin{enumerate}
  \item If $k\le m$, then
  $$
    \eta(p^k\tau)\Big|\M 10{p^m}1
   =e^{2\pi i/8}e^{-2\pi ip^{m-k}/24}\sqrt{\frac{p^m\tau+1}i}\eta(p^k\tau).
  $$
  \item If $k\ge m$, then
  $$
    \eta(p^k\tau)\Big|\M 10{p^m}1
   =e^{2\pi ip^{k-m}/24}\sqrt{\frac{p^m\tau+1}{p^{k-m}i}}\eta
    \left(\frac{p^m\tau+1}{p^{k-m}}\right).
  $$
  \end{enumerate}
\end{Lemma}

\begin{proof}
  If $k\le m$, we have
  $$
    p^k\frac{\tau}{p^m\tau+1}=\M 10{p^{m-k}}1 p^k\tau.
  $$
  Hence, by Lemma \ref{lemma: eta}, we find that
  \begin{equation*}
  \begin{split}
    \eta(p^k\tau)\Big|\M 10{p^m}1
  &=\epsilon(1,0,p^{m-k},1)\sqrt{\frac{p^m\tau+1}{i}}\eta(p^k\tau) \\
  &=i^{(1-p^{m-k})/2}e^{2\pi ip^{m-k}/12}\sqrt{\frac{p^m\tau+1}i}\eta(p^k\tau),
  \end{split}
  \end{equation*}
  which yields the first statement of the lemma.

  If $k\ge m$, we have
  $$
    p^k\frac{\tau}{p^m\tau+1}=\M{p^{k-m}}{-1}10\frac{p^m\tau+1}{p^{k-m}}.
  $$
  By Lemma \ref{lemma: eta} again, we have
  \begin{equation*}
  \begin{split}
    \eta(p^k\tau)\Big|\M 10{p^m}1
  &=\epsilon(p^{k-m},-1,1,0)\sqrt{\frac{p^m\tau+1}{p^{k-m}i}}
    \eta\left(\frac{p^m\tau+1}{p^{k-m}}\right) \\
  &=e^{2\pi ip^{k-m}/24}\sqrt{\frac{p^m\tau+1}{p^{k-m}i}}
    \eta\left(\frac{p^m\tau+1}{p^{k-m}}\right).
  \end{split}
  \end{equation*}
  This proves the lemma.
\end{proof}

\begin{Remark}
From Lemmas \ref{lemma: uniformizers} and \ref{lemma: eta at cusps},
we can easily deduce the orders of $\eta(p^k\tau)$ at each cusp,
recovering the results in Corollary \ref{corollary: order of eta}.
\end{Remark}

\subsection{}
We now use Lemma \ref{lemma: eta at cusps} to obtain the leading
coefficients of modular functions at cusps. Here for an odd prime $p$,
we let
\begin{equation}\label{def:p-ast}
  p^\ast=e^{2\pi i(p-1)/4}p, \qquad
  \sqrt{p^\ast}=e^{2\pi i(p-1)/8}\sqrt p.
\end{equation}

\begin{Proposition}\label{prop:leading-coeff} 
Assume that $p^n$ is a prime power with $p\ge 5$
and $n\ge 1$. 
Let
  $$
    f(\tau)=\left(\frac{\eta(p\tau)}{\eta(\tau)}\right)^{24/(p-1,12)}, \qquad
    g_k(\tau)=\frac{\eta(p^{k+2}\tau)}{\eta(p^k\tau)}, \quad k=0,\ldots,n-2,
  $$
  be the modular functions defined in Proposition \ref{proposition:
    generators of P}.
  \begin{enumerate}
  \item
  If $m\ge n/2$, then the leading Fourier coefficients of $f(\tau)$
  and $g_k(\tau)$ with respect to the local uniformizer at $1/p^m$
  chosen using \eqref{equation: choice of sigma} are
  $$
    \begin{cases}
    1 &\text{for }f(\tau), \\
    1 &\text{for }g_k(\tau)\text{ with }k\le m-2, \\
    (-1)^{(p-1)/2}e^{-2\pi iab/p}/\sqrt{p^\ast} &\text{for }g_{m-1}(\tau), \\
    e^{-2\pi iab/p^{k+2-m}}/p &\text{for }g_k(\tau)\text{ with }k\ge
    m. \end{cases}
  $$
  \item If $m<n/2$, then the leading Fourier coefficients of $f(\tau)$
    and $g_k(\tau)$ with respect to the local uniformizer at $-1/p^m$
  chosen using \eqref{equation: choice of sigma} are
  $$
    \begin{cases}
    p^{-12/(p-1,12)} &\text{for }f(\tau)\text{ when }m=0, \\
    e^{-2\pi ia/p^m} &\text{for }f(\tau)\text{ when }m\ge1, \\
    1/p &\text{for }g_k(\tau)\text{ with }k\ge m, \\
    e^{2\pi iab/p}/\sqrt{p^\ast} &\text{for }g_{m-1}(\tau), \\
    e^{2\pi iab/p^{m-k}} &\text{for }g_k(\tau) \text{ with }k\le m-2.
    \end{cases}
  $$
  \end{enumerate}
\end{Proposition}

\begin{proof} Consider the function $f(\tau)$ first. Since $n$ is
  assumed to be at least $1$, when $m\ge n/2$, we have $m\ge 1$ and
  the first part of Lemma \ref{lemma: eta at cusps} applies. We find
  that
  \begin{equation*}
  \begin{split}
    \frac{\eta(p\tau)}{\eta(\tau)}\Big|\M10{p^m}1
  &=e^{-2\pi ip^{m-1}(p-1)/24}\frac{\eta(p\tau)}{\eta(\tau)}.
  \end{split}
  \end{equation*}
  It follows that
  $$
    f(\tau)\Big|\M10{p^m}1=f(\tau)
  $$
  and the leading Fourier coefficient is $1$. Similarly, if $k$ is
  less than or equal to $m-2$, then the leading Fourier coefficient of
  $g_k(\tau)$ at the cusp $1/p^m$ is $1$.

  If $k=m-1$, then $k<m$ and $k+2=m+1>m$. By Lemma \ref{lemma: eta at
    cusps}, we have
  $$
    \eta(p^{m+1}\tau)\Big|\M10{p^m}1
   =e^{2\pi ip/24}\sqrt{\frac{p^m\tau+1}{pi}}\eta
    \left(\frac{p^m\tau+1}p\right),
  $$
  $$
    \eta(p^{m-1}\tau)\Big|\M10{p^m}1
   =e^{2\pi i/8}e^{-2\pi ip/24}\sqrt{\frac{p^m\tau+1}i}\eta(p^{m-1}\tau),
  $$
  and the leading coefficient of $g_{m-1}(\tau)$ at $1/p^m$ is
  \begin{equation*}
  \begin{split}
    \frac1{\sqrt p}e^{2\pi i(p/12-1/8+1/24p)}
  &=\frac1{\sqrt p}e^{2\pi i(p-1)/8}e^{-2\pi i(p^2-1)/24p} \\
  &=\frac{(-1)^{(p-1)/2}}{\sqrt{p^\ast}}e^{-2\pi iab/p}.
  \end{split}
  \end{equation*}
  (Here we remind the reader that the leading coefficient of
  $\eta((c\tau+d)/e)$ is $e^{2\pi id/24e}$.)

  When $k\ge m$, by Part (2) of Lemma \ref{lemma: eta at cusps},
  $$
    \frac{\eta(p^{k+2}\tau)}{\eta(p^k\tau)}\Big|\M10{p^m}1
   =\frac1pe^{2\pi ip^{k-m}(p^2-1)/24}
    \frac{\eta((p^m\tau+1)/p^{k+2-m})}{\eta((p^m\tau+1)/p^{k-m})},
  $$
  whose leading Fourier coefficient is
  \begin{equation*}
  \begin{split}
    \frac1pe^{2\pi i(1-p^2)/24p^{k+2-m}}
   =\frac1pe^{-2\pi iab/p^{k+2-m}}.
  \end{split}
  \end{equation*}
  This completes the proof of the case of $m\ge n/2$.

  We next consider the case $\alpha_m=-1/p^m$ with $m<n/2$. Recall that
  the choice of $\sigma_{m}$ is given in
  \eqref{equation: choice of sigma}. Noticing that
  $$
    \eta(p^k\tau)\Big|\M0{-1}{p^n}0=\eta(-1/p^{n-k}\tau)
   =\sqrt{\frac{p^{n-k}\tau}i}\eta(p^{n-k}\tau),
  $$
  we have
  $$
    g_k(\tau)\Big|\sigma_{m}
   =\frac{\eta(p^{k+2}\tau)}{\eta(p^k\tau)}\Big|\M0{-1}{p^n}1
    \Big|\sigma_{n-m}
   =\frac1{pg_{n-k-2}(\tau)}\Big|\sigma_{n-m}.
  $$
  Thus, using the results in Part (1), we find that the leading
  coefficients of $g_k(\tau)$ at the cusp $\alpha_m=-1/p^m$ are
  $$
    \begin{cases}
    1/p, &\text{if }k\ge m, \\
    e^{2\pi iab/p}/\sqrt{p^\ast}, &\text{if }k=m-1, \\
    e^{2\pi iab/p^{m-k}} &\text{if }k\le m-2.
    \end{cases}    
  $$
  Finally, for the function $f(\tau)$, we have
  $$
    f(\tau) \Big|\sigma_{m}
   =\frac1{p^{12/(p-1,12)}}
    \left(\frac{\eta(p^{n-1}\tau)}{\eta(p^n\tau)}\right)^{24/(p-1,12)}
    \Big|\sigma_{n-m}.
  $$
  When $m\ge 1$, by Part (2) of Lemma \ref{lemma: eta at cusps},
  $$
    \left(\frac{\eta(p^{n-1}\tau)}{\eta(p^n\tau)}\right)^{24/(p-1,12)}
    \Big|\sigma_{n-m}
   =p^{12/(p-1,12)}\left(\frac{\eta((p^{n-m}\tau+1)/p^{m-1})}
    {\eta((p^{n-m}\tau+1)/p^m)}\right)^{24/(p-1,12)}.
  $$
  Thus, the leading coefficient of $f(\tau)$ at $\alpha_m$ is
  $$
    \left(e^{2\pi ip^{-m}(1-p)/24}\right)^{24/(p-1,12)}
   =e^{-2\pi ia/p^m}.
  $$
  When $m=0$, by Part (1) of Lemma \ref{lemma: eta at cusps},
  $$
    \left(\frac{\eta(p^{n-1}\tau)}{\eta(p^n\tau)}\right)^{24/(p-1,12)}
    \Big|\sigma_{n}
   =\left(e^{2\pi i(1-p)/24}\frac{\eta(p^{n-1}\tau)}{\eta(p^n\tau)}
    \right)^{24/(p-1,12)}.
  $$
  Thus, the leading coefficient of $f(\tau)$ at $\alpha_0$ is $p^{-12/(p-1,12)}$.
  This completes the proof of the proposition.
\end{proof}

%

\section{Proof of Theorem \ref{thm:main}}\label{sect:proof-pn}
\subsection{}
We keep to assume $N=p^n$, 
where $p \ge 5$ is a prime and $n$ is a positive integer.
As in \S \ref{sect:cuspidal-pn},
we write $P_i:=Q_{p^i}$ for the cusp of level $p^i$
for $i=0, \dots, n$ (see \S \ref{sect:cusps}).
The residue field of $P_i$
is given by 
\[\Q(P_i) = \Q(\zeta_{p^{d(i)}}),
\quad d(i)=\min(i, n-i).
\]
Our task is to show
the injectivity of the composition map
\begin{equation}\label{eq:to-be-inj}
 \CC(p^n) \hookrightarrow J(\Q)_{\Tor} \overset{\delta}{\to} 
\bigoplus_{i=0}^{n-1} \Q(P_i) \otimes \Q/\Z,
\end{equation}
where $\delta$ is the map from \eqref{eq:ex-seq1}.

\begin{lemma}\label{lem:def-s}
Let $p \geq 3$ be a prime and $m \in \Z_{>0}$.
Then the maps
\begin{gather*}
\Q/\Z \to \Q^\times \otimes \Q/\Z, ~ x \mapsto p \otimes x,
\\
\Q/\Z \to \Q(\zeta_{p^m})^\times \otimes \Q/\Z, 
~ x \mapsto \sqrt{p^*} \otimes x,
\end{gather*}
are (split) injections.
(See \eqref{def:p-ast} for the definition of $\sqrt{p^*}$.)
\end{lemma}
\begin{proof}
Splitting is automatic because $\Q/\Z$ is injective.
The first statement follows from the elementary fact that 
$\Q^\times$ is the direct sum of $\{ \pm 1\}$ and
the free abelian group on the set of all prime numbers.
From this, the second statement 
is reduced to showing
\begin{equation}\label{eq:p-st}
 \ker[\Q^\times \otimes (\Q/\Z)
 \to \Q(\zeta_{p^m})^\times \otimes (\Q/\Z)]
= \{ 0, p^* \otimes \frac{1}{2} \}.
\end{equation}
The right hand side is contained in the left,
because $\sqrt{p^*} \in \Q(\zeta_{p^m})$.
Thus it suffices to show the left hand side of
\eqref{eq:p-st} is of order $2$.
Put $\Q/\Z(1):=\mu(\ol{\Q})$.
In terms of Galois cohomology,
this group can be rewritten  as
\[ \ker[H^1(\Q, \Q/\Z(1))\to H^1(\Q(\zeta_{p^m}), \Q/\Z(1))],
\]
which is then identified with
\[ H^1(G, H^0(\Q(\zeta_{p^m}), \Q/\Z(1))),
\quad G=\Gal(\Q(\zeta_{p^m})/\Q)
\]
by the inflation-restriction sequence.
Note that
$H^0(\Q(\zeta_{p^m}), \Q/\Z(1))=\mu_{2p^m}$. 
Since $G \cong (\Z/p^m\Z)^\times$ is a cyclic group,
the order of 
$H^1(G, \mu_{2p^m})$
agrees with that of
the Tate cohomology
\[
\hat{H}^0(G, \mu_{2p^m}):=\frac{\{ x \in \mu_{2p^m} \mid \sigma(x)=x
~\text{for all}~ 
\sigma \in G \}}
{\{ \prod_{\sigma \in G}\sigma(x) \mid x \in \mu_{2p^m} \}}.
\]
Direct computation shows that
$\hat{H}^0(G, \mu_{2p^m})$ is of order two.
This completes the proof of the lemma.
\end{proof}

\subsection{}
Let $\Lambda_i$ be the subgroup of $\Q(P_i)^\times/\mu(\Q(P_i))$
generated by $p$ (resp. $\sqrt{p^*}$)
for $i=0$ (resp. for $i=1, \dots, n-1$).
We also let
\[ \Lambda := \bigoplus_{i=0}^{n-1} \Lambda_i 
\subset \bigoplus_{i=0}^{n-1} \Q(P_i)^\times/\mu(\Q(P_i)).
\]
Proposition \ref{prop:leading-coeff} and Lemma \ref{lem:evaluation-delta}
show that the map $\delta$ in \eqref{eq:to-be-inj}
factors as
\begin{equation}\label{eq:delta-factor}
\xymatrix{
\sC(p^n) 
\ar[r]^{\widetilde{\delta}}
\ar[rd]_{\eqref{eq:to-be-inj}}
& 
\Lambda \otimes \Q/\Z
\ar@{_{(}->}[d]
\\
& \bigoplus_{i=0}^{n-1} \Q(P_i)^\times \otimes \Q/\Z,
}
\end{equation}
where the right vertical injection is provided by Lemma \ref{lem:def-s}.
We are reduced to showing the injectivity of $\widetilde{\delta}$.

We choose a uniformizer $t_{P_i}$
described in Lemma \ref{lemma: uniformizers}
at each cusp $P_i$.
Using them,
we define a homomorphism
\begin{equation}\label{eq:Delta}
 \Delta : \sP(p^n) \to \Lambda
\end{equation}
by,
for any $h \in K^\times$ such that $\div(h)$ is supported on $D$,
\[ 
\Delta(\div(h)) =
\Big( (\frac{h}{t_{P_n}^{\ord_{P_n}(h)}})(P_n) 
(\frac{t_{P_i}^{\ord_{P_i}(h)}}{h})(P_i) \Big)_{i=0}^{n-1}
\in \Lambda \subset \bigoplus_{i=0}^{n-1} \Q(P_i)^\times/\mu(\Q(P_i)).
\]
Note that the image of $\Delta$ is contained in $\Lambda$
by 
Proposition \ref{prop:leading-coeff}.
Note also that,
unlike $\delta$, this map depends on
our choice of uniformizers $t_{P_i}$.

Recall that $a=(p-1)/(p-1, 12)$. Put $a'=12/(p-1, 12)$.

\begin{lemma}\label{eq:lemma-final}
\begin{enumerate}
\item 
The map $\Delta$ is injective.
\item 
The cokernel of $\Delta$ is a cyclic group of order $a'$
generated by the class of $\lambda := p \in \Lambda_0 \subset \Lambda$.
\item
There exist $c_0, \dots, c_{n-2} \in \Z$ such that
$a'\lambda = \Delta(\div(fg_0^{c_0} \dots g_{n-2}^{c_{n-2}}))$.
Here $f, g_0, \dots, g_{n-2}$ are functions introduced in 
Proposition \ref{proposition: generators of P}.
%
\end{enumerate}
\end{lemma}
\begin{proof}
Recall from Corollary \ref{cor:gen} (1)
that $\div(f), \div(g_0), \dots, \div(g_{n-2})$ 
form a $\Z$-basis of $\sP(p^n)$.
Let us take a $\Z$-basis of $\Lambda =\oplus_{i=0}^{n-1} \Lambda_i$
given by the generators 
$\lambda=p \in \Lambda_0$ and $\sqrt{p^*} \in \Lambda_i$ for $i=1, \dots, n-1$.
Then Proposition \ref{prop:leading-coeff} shows that
the map $\Delta$ is represented by
\[
-\begin{pmatrix}
a' & 0 & 0 & 0 &  \dots & 0
\\
1 & 1 & 0 & 0 &  \dots & 0
\\
1 & 2 & 1 & 0 &  \dots & 0
\\
1 & 2 & 2 & 1 &  \dots & 0
\\
\vdots &  &  &  &  & \vdots
\\
1 & 2 & \dots & 2 & 2 & 1
\end{pmatrix}
\]
from which the lemma follows.
\end{proof}

\subsection{Proof of Theorem \ref{thm:main}}

Recall from Lemma \ref{eq:lemma-final} (1)
that the map $\Delta$ defined in \eqref{eq:Delta} is injective.
Since $\PP(p^n)$ is of finite index in $\DD(p^n)$
(see \eqref{eq:def-pp}),
$\Delta$ has a unique extension 
\[ \widetilde{\Delta} : \DD(p^n) \to \Lambda \otimes \Q.
\]
which is also injective.
By Lemma \ref{lem:evaluation-delta},
we have a commutative diagram with exact rows
\[
\xymatrix{
0 \ar[r]
&
\sP(p^n) \ar[d]^{\Delta} \ar[r]
&
\sD(p^n) \ar[d]^{\widetilde{\Delta}} \ar[r]
&
\sC(p^n) \ar[d]^{\widetilde{\delta}} \ar[r]
&
0
\\
0 \ar[r]
&
\Lambda \ar[r]
&
\Lambda \otimes \Q \ar[r]
&
\Lambda \otimes \Q/\Z \ar[r]
&
0.
}
\]
where $\widetilde{\delta}$ is from \eqref{eq:delta-factor}.
We get an exact sequence
\begin{equation}\label{eq:snake}
 0 \to \ker(\widetilde{\delta}) 
\to \Coker(\Delta) \overset{\psi}{\to} \Coker(\widetilde{\Delta}).
\end{equation}
It remains to show $\psi$ is injective.
In view of Lemma \ref{eq:lemma-final} (2),
this amounts to showing that
the image of $\lambda=p \in \Lambda_0 \subset \Lambda$ 
in $\Coker(\widetilde{\Delta})$ has order $a'$.
Let $b>0$ be a divisor of $a'$ 
such that
the image of $b \lambda$ vanishes in $\Coker(\widetilde{\Delta})$.
This means that 
$\div(f g_0^{c_0} \dots g_{n-2}^{c_{n-2}}) = (a'/b) \sE$
for some $\sE \in \sD(p^n)$,
where $c_0, \dots, c_{n-2} \in \Z$ 
are taken from Lemma \ref{eq:lemma-final} (3).
Since $a$ and $a'$ are relatively prime to each other,
Corollary \ref{cor:gen} (2) shows that $b$ must be $a'$.
This completes the proof.
\qed

\section{The case of $N=pq$}\label{sect:level-pq}
\subsection{}
In this section, we present an outline of the proof of
Proposition \ref{prop:intro-pq}.
Since the proof goes similarly with Theorem \ref{thm:main},
we will be brief and omit details.
Let $N=pq$ where
$p, q$ are two distinct prime numbers such that $p \equiv q \equiv 1 \bmod 12$.
We use Takagi's result \cite[Theorem 5.1]{Takagi}
that determines the order of $\sC(N)$ for any square free $N$.
Here we state it in the special case $N=pq$, $p,q\equiv 1\mod 12$:
\begin{proposition}[Takagi]\label{prop:takagi}
The order of $\sC(pq)$ is given by $4abc$, where
\[ a=\frac{(p-1)(q+1)}{24}, ~
   b=\frac{(p+1)(q-1)}{24}, ~
   c=\frac{(p-1)(q-1)}{24}.
\]
\end{proposition}

\subsection{Proof of Proposition \ref{prop:intro-pq}}
There are exactly four cusps on $X_0(pq)$,
all of which are $\Q$-rational.
Their levels are $1, p, q, pq$ (see \S \ref{sect:cusps}).
We give their names as follows:
\[ C_0(pq)=\{ P_0=Q_1,~ P_1=Q_p, P_2=Q_q, P_3=Q_{pq} \} \]
so that we have a $\Z$-basis of $\DD(pq)$ given by
\[ D_1=P_0-P_3, ~D_2=P_1-P_3, ~D_3=P_2-P_3. \]
Using Proposition \ref{prop:takagi},
one sees that
the group $\PP(pq)$ is generated by the divisors of
\[
f_1 = \frac{\eta(\tau)\eta(q\tau)}{\eta(p\tau)\eta(pq\tau)}, ~~
f_2 = \frac{\eta(\tau)\eta(p\tau)}{\eta(q\tau)\eta(pq\tau)}, ~~
f_3 = \frac{\eta(\tau)\eta(pq\tau)}{\eta(p\tau)\eta(q\tau)}
\]
by an argument similar to Proposition \ref{proposition: generators of P}.
Lemma \ref {lem:eta-prod-div} shows that
the divisors of the functions $f_1, f_2, f_3$ are respectively given by
\[ 
a(D_1-D_2+D_3),~~
b(D_1+D_2-D_3),~~
c(D_1-D_2-D_3).
\]
Let $m \in \{ 1, p, q, pq \}$ and put $m'=pq/m$.
Choose integers $b$ and $d$ such that $dm'-bm=1$.
Put
\[
\sigma_m:=\begin{pmatrix} m' & -b \\ pq & dm' \end{pmatrix}.
\]
Then $\sigma_m$ normalizes $\Gamma_0(pq)$ and
satisfies $\sigma_m \infty=1/m$.
We may take $e^{2 \pi i \sigma_m^{-1}\tau}$
as a local uniformizer at the cusp of level $m$
(cf. Lemma \ref{lemma: uniformizers}).
With this choice, the leading coefficients of $f_1, f_2, f_3$, up to
$\pm 1$ signs, are given by the following table:
\[
\begin{array}{|c|cccc|} \hline
 & P_0 & P_1 & P_2 & P_3 
\\
\hline
f_1 & p & 1 & p & 1  \\
f_2 & q & q & 1 & 1  \\
f_3 & 1 & 1 & 1 & 1 \\
\hline
\end{array}
\]
Finally, by using Lemma \ref{lem:evaluation-delta},
we find the kernel of 
$\sC(pq) \hookrightarrow J_0(pq)(\Q)_{\Tor} \overset{\delta}{\to}$
is a cyclic group of order $c$ generated by the class of $D_1-D_2-D_3$.
\qed

\bibliographystyle{plain}

\end{document}